\documentclass[a4paper,12pt]{article}
\usepackage{amsmath,amsfonts,amssymb,graphics,amsthm}
\usepackage[dvips]{graphicx}
\usepackage{subfigure}

 \usepackage[usenames,dvipsnames]{pstricks}
 \usepackage{epsfig}
 \usepackage{pst-grad} 
 \usepackage{pst-plot} 

\usepackage{a4wide}

\newtheorem{theorem}{Theorem}[section] \newtheorem{prop}[theorem]{Proposition}
\newtheorem{lemma}[theorem]{Lemma} \newtheorem{defn}[theorem]{Definition}
\newtheorem{cor}[theorem]{Corollary}

\newcommand{\Integer}{\mathbb{Z}}

\newcommand{\Z}{\Integer}

\newcommand{\Q}{\mathbb{Q}}
\newcommand{\R}{\mathbb{R}}

\def\cT{\mathcal{T}}

\def\cN{\mathcal{N}}

\def\cI{\mathcal{I}}
\def\cH{\mathcal{H}}

\def\cF{\mathcal{F}}

\def\cA{\mathcal{A}}

\def\pd{\partial}

\def\P{\mathbb{P}}
\def\E{\mathbb{E}}
\def\eps{\varepsilon}

\DeclareMathOperator{\var}{Var}
\DeclareMathOperator{\cov}{Cov}
\DeclareMathOperator{\diam}{Diam}
\DeclareMathOperator{\Leb}{Leb}

\newcommand{\indic}[1]{\mathbf{1}_{\{#1\}}}
\parskip \medskipamount

\title{Diffusion in planar Liouville quantum gravity}
\author{Nathana\"el Berestycki\footnote{Statistical Laboratory, University of Cambridge. Research supported in part by EPSRC grants EP/GO55068/1 and
    EP/I03372X/1.}}
\date{\today}
\begin{document}
\maketitle

\abstract{
We construct the natural diffusion in the random geometry of planar Liouville quantum gravity. Formally, this is the Brownian motion in a domain $D$ of the complex plane for which the Riemannian metric tensor at a point $z \in D$ is given by $\exp ( \gamma h(z) - \frac12 \gamma^2 \E ( h(z)^2) )$. Here $h$ is an instance of the Gaussian Free Field on $D$ and $\gamma \in (0,2)$ is a parameter. 
We show that the process is almost surely continuous and enjoys certain conformal invariance properties. We also estimate the Hausdorff dimension of times that the diffusion spends in the thick points of the Gaussian Free Field, and show that it spends Lebesgue-almost all its time in the set of $\gamma$-thick points, almost surely. 

The diffusion is constructed by a limiting procedure after regularisation of the Gaussian Free Field.  The proof is inspired by arguments of Duplantier--Sheffield for the convergence of the Liouville quantum gravity measure, previous work on multifractal random measures, and relies also on estimates on the occupation measure of planar Brownian motion by Dembo, Peres, Rosen and Zeitouni. 

A similar but  deeper result has been independently and simultaneously proved by Garban, Rhodes and Vargas.
}

\newpage

\section{Introduction}

This paper is motivated by a recent series of works on planar Liouville quantum gravity and the so-called KPZ relation (for Knizhnik, Polyakov and Zamolodchikov). The KPZ relation describes a way to relate geometric quantities associated with Euclidean models of statistical physics to their formulation in a setup governed by a certain \emph{random geometry}, the so-called Liouville (critical) quantum gravity. This is a problem which has a long and distinguished history and for which we refer the interested reader to the recent breakthrough paper of Duplantier and Sheffield \cite{DuplantierSheffield} and the excellent survey article by Garban \cite{Garban}.

A central problem in this area
is the construction of a natural random metric in the plane, enjoying properties of conformal invariance, such that a KPZ relation holds. By this we mean that given a set $A$ in the plane, the Hausdorff dimensions of $A$ endowed either with the Euclidean metric or the random (quantum) metric are related by a deterministic transformation. Given these requirements, it is reasonably natural to look for or postulate that the local metric at a point $z$ can be written in the form $\exp ( \gamma h(z)  - (\gamma^2/2) \E h(z)^2)$, where $h$ is a Gaussian Free Field and $\gamma$ is a parameter. Unfortunately, $h$ is not a function but a random distribution, and the exponential of a distribution is not in general well-defined.

While making sense of this notion of random metric is still wide open, Duplantier and Sheffield, in the paper mentioned above, were able to construct a random measure, called the quantum gravity measure, which intuitively speaking corresponds to the volume measure of the metric. Remarkably, using this measure, they were able to define suitable notions of scaling dimensions for a set $A$ and show that a KPZ relation holds, where the deterministic transformation involves a quadratic polynomial.

\medskip The purpose of this paper is to show that a natural notion of \emph{diffusion} also makes sense in this context. Roughly speaking, one can summarise the main result by saying that, while we still don't know how to measure the distance between two points $z$ and $w$, it is possible to say how long it would take a Brownian motion to go from $z$ to $w$. The key idea is to note that, using conformal invariance of Brownian motion in two dimensions, it suffices to parametrise the Brownian motion correctly.

\medskip \noindent \textbf{Important note.} As I was preparing this paper, I learnt that Garban, Rhode and Vargas were working on a similar problem. Their paper has now appeared on arxiv \cite{GarbanRhodesVargas}. Their approach seems more powerful than the one here.

\subsection{Looking for the right object}

What follows is an informal discussion which is aimed to explain where the definition comes from. By \emph{local metric} $\rho(z)$ at a point $z \in D$, we mean that small segments of  Euclidean length $\eps$ are in the Riemannian metric considered to have distance $\rho(z) \eps$ at the first order when $\eps \to 0$. 

Let $U,D$ be two proper simply connected domains, and let $f: U \to D$ be a conformal isomorphism. We think of $U$ as being a (wild) domain endowed with the random geometry, and $D$ a nice domain such as the unit disc, in which we read this geometry.  If $(W_t, t \ge 0)$ is a standard Brownian motion in $U$ (i.e., stopped upon leaving $U$), then we simply wish to describe how $(W_t, t\ge 0)$ is parametrized by $D$. To do this, it suffices to consider $X_t = f(W_t)$. By It\^o's formula,
\begin{equation}\label{ito}
Z_t = f(W_t) = B_{\int_0^t | f'(W_s)|^2 ds};
\end{equation}
and hence $Z$ is a time-change of a Brownian motion $(B_t, t\ge 0)$ in $D$. This cannot directly be used as a definition as the time-change still involves $W$ and we only want to define the process $Z$ in terms of $B$ and the metric $\rho(z)$ in $D$ derived from mapping the metric in $U$ via $f$. Clearly, $\rho(z)$ is simply equal to $1/ |f'(w)|^2 = |g'(z)|^2 $, where $g = f^{-1}$ (see Figure \ref{Fig:RMT}). 

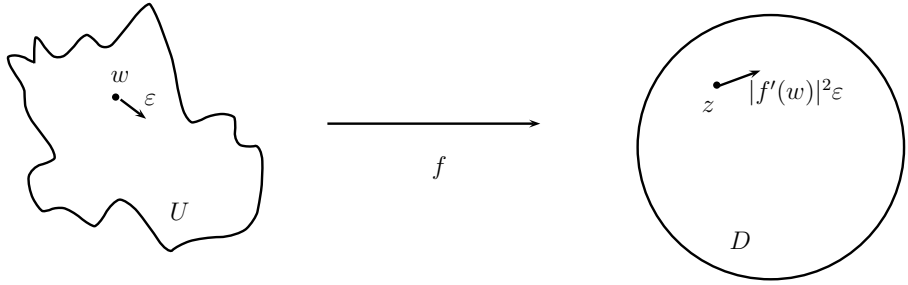
\begin{figure}\centering
\scalebox{.8} 
{
\begin{pspicture}(0,-2.29)(14.75,2.33)
\pscustom[linewidth=0.04]
{
\newpath
\moveto(0.43,1.31)
\lineto(0.43,1.49)
\curveto(0.43,1.58)(0.44,1.7)(0.45,1.73)
\curveto(0.46,1.76)(0.515,1.715)(0.56,1.64)
\curveto(0.605,1.565)(0.695,1.46)(0.74,1.43)
\curveto(0.785,1.4)(0.88,1.405)(0.93,1.44)
\curveto(0.98,1.475)(1.08,1.58)(1.13,1.65)
\curveto(1.18,1.72)(1.265,1.785)(1.3,1.78)
\curveto(1.335,1.775)(1.425,1.7)(1.48,1.63)
\curveto(1.535,1.56)(1.655,1.56)(1.72,1.63)
\curveto(1.785,1.7)(1.955,1.91)(2.06,2.05)
\curveto(2.165,2.19)(2.3,2.31)(2.33,2.29)
\curveto(2.36,2.27)(2.45,2.085)(2.51,1.92)
\curveto(2.57,1.755)(2.655,1.515)(2.68,1.44)
\curveto(2.705,1.365)(2.75,1.15)(2.77,1.01)
\curveto(2.79,0.87)(2.845,0.63)(2.88,0.53)
\curveto(2.915,0.43)(2.99,0.33)(3.03,0.33)
\curveto(3.07,0.33)(3.17,0.365)(3.23,0.4)
\curveto(3.29,0.435)(3.435,0.47)(3.52,0.47)
\curveto(3.605,0.47)(3.705,0.44)(3.72,0.41)
\curveto(3.735,0.38)(3.74,0.26)(3.73,0.17)
\curveto(3.72,0.08)(3.77,-0.03)(3.83,-0.05)
\curveto(3.89,-0.07)(4.005,-0.075)(4.06,-0.06)
\curveto(4.115,-0.045)(4.175,-0.17)(4.18,-0.31)
\curveto(4.185,-0.45)(4.175,-0.695)(4.16,-0.8)
\curveto(4.145,-0.905)(4.115,-1.08)(4.1,-1.15)
\curveto(4.085,-1.22)(3.995,-1.36)(3.92,-1.43)
\curveto(3.845,-1.5)(3.625,-1.57)(3.48,-1.57)
\curveto(3.335,-1.57)(3.085,-1.605)(2.98,-1.64)
\curveto(2.875,-1.675)(2.735,-1.75)(2.7,-1.79)
\curveto(2.665,-1.83)(2.55,-1.745)(2.47,-1.62)
\curveto(2.39,-1.495)(2.22,-1.26)(2.13,-1.15)
\curveto(2.04,-1.04)(1.89,-0.94)(1.83,-0.95)
\curveto(1.77,-0.96)(1.655,-1.075)(1.6,-1.18)
\curveto(1.545,-1.285)(1.455,-1.39)(1.42,-1.39)
\curveto(1.385,-1.39)(1.27,-1.32)(1.19,-1.25)
\curveto(1.11,-1.18)(0.97,-1.1)(0.91,-1.09)
\curveto(0.85,-1.08)(0.75,-1.09)(0.71,-1.11)
\curveto(0.67,-1.13)(0.6,-1.145)(0.57,-1.14)
\curveto(0.54,-1.135)(0.46,-1.03)(0.41,-0.93)
\curveto(0.36,-0.83)(0.325,-0.65)(0.34,-0.57)
\curveto(0.355,-0.49)(0.455,-0.39)(0.54,-0.37)
\curveto(0.625,-0.35)(0.735,-0.3)(0.76,-0.27)
\curveto(0.785,-0.24)(0.74,-0.115)(0.67,-0.02)
\curveto(0.6,0.075)(0.47,0.25)(0.41,0.33)
\curveto(0.35,0.41)(0.285,0.56)(0.28,0.63)
\curveto(0.275,0.7)(0.27,0.805)(0.27,0.84)
\curveto(0.27,0.875)(0.22,0.98)(0.17,1.05)
\curveto(0.12,1.12)(0.05,1.235)(0.03,1.28)
\curveto(0.01,1.325)(-0.0050,1.39)(0.0,1.41)
\curveto(0.0050,1.43)(0.07,1.445)(0.13,1.44)
\curveto(0.19,1.435)(0.275,1.41)(0.3,1.39)
\curveto(0.325,1.37)(0.37,1.345)(0.43,1.33)
}
\usefont{T1}{ptm}{m}{n}
\rput(2.82,-1.145){$U$}
\psdots[dotsize=0.12](1.77,0.75)
\psline[linewidth=0.04cm,arrowsize=0.05291667cm 2.0,arrowlength=1.4,arrowinset=0.4]{->}(5.25,0.31)(8.75,0.31)
\pscircle[linewidth=0.04,dimen=outer](12.54,-0.08){2.21}
\psdots[dotsize=0.12](11.65,0.95)
\usefont{T1}{ptm}{m}{n}
\rput(12.04,-1.625){$D$}
\usefont{T1}{ptm}{m}{n}
\rput(11.51,0.595){$z$}
\psline[linewidth=0.04cm,arrowsize=0.05291667cm 2.0,arrowlength=1.4,arrowinset=0.4]{->}(1.85,0.71)(2.27,0.39)
\psline[linewidth=0.04cm,arrowsize=0.05291667cm 2.0,arrowlength=1.4,arrowinset=0.4]{->}(11.67,0.93)(12.37,1.19)
\usefont{T1}{ptm}{m}{n}
\rput(1.84,1.095){$w$}
\usefont{T1}{ptm}{m}{n}
\rput(2.34,0.715){$\eps$}
\usefont{T1}{ptm}{m}{n}
\rput(12.95,0.835){$|f'(w)|^2\eps$}
\usefont{T1}{ptm}{m}{n}
\rput(7.09,-0.425){$f$}
\end{pspicture} 
}
\caption{Local metric and conformal map}
\label{Fig:RMT}
\end{figure}

The reader can then easily check that setting
\begin{equation}
  \label{Dcorrect}
Z_t = B(\mu_t^{-1}) ; \text{ where } \mu_t = \int_0^t |g'(B_s)|^2 ds,
\end{equation}
and $\mu_t^{-1} : = \inf \{s>0: \mu_s >t\}$, gives the same process as \eqref{ito}. The advantadge of this way of writing $Z$ is that it involves only the standard Brownian motion $(B_t,t \ge 0)$ in the nice domain $D$ and the local metric $\rho(z)$ at any point $z \in D$, which we assume to be given.

\subsection{Statements}
We will thus use \eqref{Dcorrect} as our definition. Fix a proper connected domain $D \subset \mathbb{C}$, and let $h$ be an instance of the Gaussian Free Field in $D$.
(We use the Duplantier-Sheffield normalisation of the Green function). Formally, 
$h$ is a centered Gaussian process indexed by the Sobolev space $H_0^1(D)$, which is the completion of $C^\infty_{K}(D)$ with respect to the scalar product
$$
(f,g) = \frac1{2\pi} \int_{D}(\nabla f )\cdot(\nabla g). 
$$
Then $h$ is a centered Gaussian process such that if $(h,f)_{\nabla}$ is the value of the field at the function $f \in H_0^1(D)$, then
$$
\cov [(h,f)_{\nabla} , (h,g)_{\nabla}] = (f,g)_{\nabla}
$$
by definition.
%

For $z \in D$ and $\eps$ sufficiently small, we let $h_\eps(z)$ be the well-defined average of $h$ over a circle of radius $\eps$ about $x$. (We refer the reader to \cite{Sheffield} for a proof that this is indeed well-defined and other general facts about the Gaussian Free Field). 
Then we define a process $(Z_{\eps}(t),t\ge 0)$ as in \eqref{Dcorrect}. That is, let $z \in U$ and let $(B_t, t \ge 0)$ be a planar Brownian motion such that $Z_0 = z$ almost surely. We put
$$
Z_\eps(t) = B( \mu_\eps^{-1}(t)), \ \ \text{ where } \mu_\eps(t) = \int_0^{t\wedge T} e^{\gamma h_\eps(B_s) - \frac{\gamma^2}2 \var h_\eps(B_t)} ds,
$$
and $\mu_\eps^{-1}(t ) = \inf \{s \ge 0: \mu_\eps(s) > t\}$, $T = \inf \{ t \ge 0: B_s \notin D\}$.

\begin{defn}The Liouville diffusion, if it exists, is the limit as $\eps \to 0$ of the process $Z_\eps$.\end{defn}

Obviously, since $B$ does not depend on $\eps$, the issue of convergence of the process $Z_\eps$ reduces to that of the clock process $(\mu_\eps(t), t \le T)$.
\begin{theorem}\label{T:BM}
Assume $0\le \gamma< 2$. Then $(Z_\eps(t), t \ge 0)$ converges almost surely as $\eps \to 0$ to a random process $(Z(t), t\ge 0)$ which is almost surely continuous up to the hitting time of $\pd D$.
\end{theorem}

We now address conformal invariance properties. Let $D, \tilde D$ be two simply connected domains and let $\phi:D \to \tilde D$ be a conformal isomorphism (a bijective conformal map with conformal inverse). 

\begin{theorem}
Let $Q = \frac{\gamma}2 + \frac2{\gamma}$
and $\psi = \phi^{-1}$. Then we can write
$$
\phi(B_{\mu^{-1}(t)} )= \tilde B_{\tilde \mu_\psi^{-1}(t)}
$$
where $\tilde B$ is a Brownian motion in $\tilde D$,  $$\tilde \mu_\psi (t) =  \int_0^{t } \eps^{\gamma^2 /2 }  e^{\gamma  [ \tilde h_\eps   + Q \log |\psi'|] (\tilde B_s)} ds$$
and $\tilde h$ is the Gaussian Free Field in $\tilde D$.
\end{theorem} 

In other words, mapping the Liouville diffusion $Z(t)$ by the transformation $\phi$, one obtains the corresponding Liouville diffusion in $\tilde D$, except that the Gaussian Free Field $\tilde h$ in $\tilde D$ has been replaced by $\tilde h + Q \log |\psi'|$. (This is similar to Proposition 2.1 in Duplantier--Sheffield \cite{DuplantierSheffield}.) 

Finally, it is of interest to quantify how much time the Brownian motion spends in points for which the points of the field $h$ are unusually big. 
Consider the \emph{thick points} of the Gaussian Free Field: for $\alpha>0$, let 
\begin{equation}
\label{Dthick}
\begin{cases}
\mathcal{T}^-_\alpha: &= \{z \in D: \liminf_{\eps \to 0} \frac{ h_\eps(z)}{\log (1/\eps)} \ge \alpha\},\\
\mathcal{T}^+_\alpha: &= \{z \in D: \limsup_{\eps \to 0} \frac{ h_\eps(z)}{\log (1/\eps)} \le \alpha\},
\end{cases}
\end{equation}
Hu, Miller and Peres \cite{HMP} proved that the Hausdorff dimension of $\cT_\alpha$ is a.s. $(2-\alpha^2/2)\vee 0. $

\begin{theorem}\label{T:thick}
Let $0< \gamma <2$ and let $\alpha > \gamma$. Then almost surely,
\begin{equation}\label{thickdim}
\dim\{t: Z(t) \in \cT^-_\alpha \} \le \frac{ 2 - \frac{\alpha^2}2}{2 - \alpha \gamma + \frac{\gamma^2}2}. 
\end{equation}
The same result holds when $\alpha <\gamma$ and $\cT^-_\alpha$ replaced with $\cT^+_\alpha$.
\end{theorem}

We believe but have not proved that equality holds. The upper bound is nevertheless enough to deduce the following result:
\begin{cor}
With probability one, 
$$
 \Leb  \{ t : Z(t) \notin \cT^=_\gamma\} = 0, 
$$
where
$$
 \cT^=_\alpha =  \{z \in D: \lim_{\eps \to 0} \frac{ h_\eps(z)}{\log (1/\eps)} = \alpha\}.
 $$
\end{cor}

By contrast, using the methods of Benjamini and Schramm \cite{BenjaminiSchramm} (see also Rhode and Vargas \cite{RhodesVargas}) it is possible to show the following analogue of the KPZ relation. 

\begin{prop}
Fix $A \subset  D$ a non random Borel set and let $d_0 = \dim (A)$ where $\dim$ refers to the (Euclidean) Hausdroff dimension. Then almost surely,
\begin{equation}\label{kpz}
\dim\{t: Z_t \in A\} = d 
\end{equation}
where $d$ solves the equation $d_0  + d^2 \gamma^2/2 - d(2 + \gamma^2/2) =0. $
\end{prop}

In the case of $\cT^\pm_\alpha$, as mentionned above, Hu, Miller and Peres \cite{HMP} showed that the Hausdorff dimension is $(2- \alpha^2/2)\vee 0$. Nevertheless, the formula in \eqref{kpz} does not match that from Theorem \ref{T:thick}. This is of course because $\cT^\pm_\alpha$ depends very strongly on the Gaussian Free Field. The difficulty in Theorem \ref{T:thick} is thus essentially to understand the effect on the clock process of coming near a thick point, and hence to disentangle the separate effects linked on the one hand to the trajectory of a standard Brownian motion and on the other hand to the frequency of those thick points.

\section{Proof of Theorem \ref{T:BM}: Convergence}

For the rest of the paper, with a slight abuse of notation, we call $(B_t, t\ge 0)$ a Brownian motion stopped at time $T := T_r = \inf\{t>0: \text{dist} (B_t , \pd D) \le r\}$. Here $r>0$ is a small arbitrary number. We will still call  
$$
Z_\eps(t) = B( \mu_\eps^{-1}(t)), \ \ \text{ where } \mu_\eps(t) = \int_0^{t\wedge T} e^{\gamma h_\eps(B_s) - \frac{\gamma^2}2 \var h_\eps(B_t)} ds,
$$
and $\mu_\eps^{-1}(t ) = \inf \{s \ge 0: \mu_\eps(s) > t\}$, $T = T_r$.

In this section we prove that the clock process $\mu_\eps(t)$ converges as $\eps \to 0$ to a limit (which might still be degenerate). By Proposition 3.2 in \cite{DuplantierSheffield}, 
$$
\var(h_\eps(z)) = - \log \eps + \log(R(z;D))
$$
where $R(z;D)$ is the conformal radius of $z$ in $D$. That is, $R(z;D) = 1/ |\phi'(z) |$, where $\phi:D \to \mathbb{D}$ is a conformal map such that $\phi(z) = 0$. Note that for $t\le T = T_r$, there are two nonrandom constants $c_1, c_2$ such that
$$
c_1 \le R(B_t;D)\le c_2
$$
Therefore we are interested in proving the convergence, as $\eps \to 0$, of the quantity
$$
\alpha_\eps(t) = \int_0^t e^{\bar h_\eps(B_s)}ds
$$
where $\bar h_\eps(z) = \gamma h_\eps(z) + (\gamma^2/2) \log \eps.$

\begin{proof}  We start by pointing out a potential source of confusion. Note that for each \emph{fixed} $z \in D$, the sequence $e^{\bar h_\eps(z)}$ (viewed as a function of $\eps$) forms a nonnegative martingale, in its own filtration. It follows from this and from Fubini's theorem that $\E(\alpha_\eps(t)) = t$ for all $t \ge 0$ and for all $\eps>0$. However, note that the above does \emph{not} imply that $\alpha_\eps(t)$ is a martingale as a function of $\eps$: this is because the martingale property of $e^{\bar h_\eps(z)}$ ceases to hold when the filtration contains all the information about $(h_\eps(w), w \in D)$.

Nevertheless, the random variables $\alpha_\eps(t)$ converge as $\eps \to 0$ almost surely to a limit. We now prove this statement. In fact we only prove this along the subsequence $\eps = 2^{-k}, k \ge 1$. With an abuse of notation we write $\alpha_k$ for $\alpha_{2^{-k}}$ and $h_k$ for $h_{2^{-k}}$. Then it suffices to prove that $|\alpha_k - \alpha_{k+1}| \le C r^k$ for some $r<1$ and $C<\infty$, almost surely. Assume without loss of generality that $t =1$ and let $s \in [0,1]$. Let $S_k^s = [0,1] \cap \{s + 2^{-2k} \Z\}$, and let
$$
X_k(s) = \frac1{2^{2k}} \sum_{t \in S_k^s} e^{ \bar h_{k} (B_t)}
$$
and
$$
Y_k(s) = \frac1{2^{2k}} \sum_{t \in S_k^s} e^{ \bar h_{k+1}(B_t)}.
$$
Note that
$\alpha_k = \int_0^1 X_k(s)ds$ and $\alpha_{k+1} = \int_0^1 Y_k(s) ds$ so it suffices to prove that
\begin{equation}\label{expdecay}
|X_k(s) - Y_k(s)| \le C r^k
\end{equation}
 for some $C,r<1$ uniformly in $s \in[0,1]$. As in \cite{DuplantierSheffield}, we start with the case $\gamma< \sqrt{2}$ where an easy second moment argument suffices. Let $\tilde \E(\cdot) = \E(\cdot | \sigma(B_s, s\le t))$. Then note that
\begin{align*}
 \tilde  \E(|X_k(s) - Y_k(s)|^2 ) & = \frac1{2^{4k}} \sum_{t,t' \in S_k^s} \tilde \E\left[ (e^{ \bar h_{k} (B_t)} - e^{ \bar h_{k+1} (B_t)} )( e^{ \bar h_{k} (B_{t'})} - e^{ \bar h_{k+1} (B_{t'})}) \right]
\end{align*}
Let $t, t' \in S_k^s$ and assume that $|B_t - B_{t'} | > 2^{-k}$. Then conditionally on $h_k(B_t)$ and $h_{k} (B_{t'})$, the random variables  $h_{k+1}(B_t)$ and $h_{k+1}(B_{t'})$ are independent Gaussian random variables with mean $h_k(B_t)$ (resp. $h_k(B_{t'})$) and variance $\log 2$. Thus, in that case,
\begin{align*}
  \tilde \E\left[ (e^{ \bar h_{k} (B_t)} - e^{ \bar h_{k+1} (B_t)} )( e^{ \bar h_{k} (B_{t'})} - e^{ \bar h_{k+1} (B_{t'})}) | h_k(B_t), h_{k} (B_{t'})\right]
  & = \tilde \E \left[e^{ \bar h_{k} (B_t)} - e^{ \bar h_{k+1} (B_t)} | h_k(B_t), h_{k} (B_{t'})\right] \\
  & \times \tilde \E\left[ e^{ \bar h_{k} (B_{t'})} - e^{ \bar h_{k+1} (B_{t'})}  | h_k(B_t), h_{k} (B_{t'})\right].
\end{align*}
Now, observe that $\bar h_{k+1}(B_t) = \bar h_k(B_t) + \gamma X - (\gamma^2/2) \log 2$ where $X$ is a centred Gaussian random variable with variance $\log 2$, which is independent from $\tilde B, h_k(B_t), h_k(B_{t'})$. Hence
\begin{align*}
\tilde \E \left[e^{ \bar h_{k} (B_t)} - e^{ \bar h_{k+1} (B_t)} | h_k(B_t), h_{k} (B_{t'})\right] &= e^{\bar h_k(B_t)} \left[1- e^{- (\gamma^2/2) \log 2} \tilde \E(e^{\gamma X})\right]= 0.
\end{align*}
Of course the same also holds once we uncondition on $h_k(B_t), h_{k} (B_{t'})$. It follows by Cauchy-Schwarz's inequality that
\begin{align}
 \tilde  \E(|X_k(s) - Y_k(s)|^2 ) & = \frac1{2^{4k}} \sum_{t,t' \in S_k^s} 1_{|B_t - B_{t'} | \le 2^{-k}} \tilde \E\left[ (e^{ \bar h_{k} (B_t)} - e^{ \bar h_{k+1} (B_t)} )( e^{ \bar h_{k} (B_{t'})} - e^{ \bar h_{k+1} (B_{t'})}) \right] \nonumber \\
 & \le \frac1{2^{4k}}\sum_{t,t' \in S_k^s} 1_{|B_t - B_{t'} | \le 2^{-k}}  \sqrt{\tilde \E[ (e^{ \bar h_{k} (B_t)} - e^{ \bar h_{k+1} (B_t)} )^2] \tilde \E[ (e^{ \bar h_{k} (B_{t'})} - e^{ \bar h_{k+1} (B_{t'})} )^2]}\nonumber \\
 & =  \frac{C}{2^{4k}} \sum_{t,t' \in S_k^s} 1_{|B_t - B_{t'} | \le 2^{-k}}   \E[ (e^{ \bar h_{k} (z)} - e^{ \bar h_{k+1} (z)} )^2] \label{variance}
\end{align}
To compute the expectation in the sum, we condition on $h_k(z)$ and get, letting $\eps = 2^{-k}$,
 \begin{align*}
 \E[ (e^{ \bar h_{k} (z)} - e^{ \bar h_{k+1} (z)} )^2] &\le \E(e^{2 \bar h_k(z)})
= \eps^{\gamma^2} \E(e^{2 \gamma h_k(z)}) \\
&= \eps^{\gamma^2 - 4 \gamma^2/2 } = \eps^{- \gamma^2}.
\end{align*}
Therefore,
\begin{align}\label{inter}
   \tilde  \E(|X_k(s) - Y_k(s)|^2 ) & \le C \eps^{4} \eps^{-\gamma^2} \sum_{t, t' \in S_k^s} 1_{|B_t - B_{t'} | \le 2^{-k}}.
 \end{align}

We will need the following lemma on two-dimensional Brownian motion:
\begin{lemma}\label{L:dprz}Uniformly in $s \in [0,1]$:
$$
\sum_{t, t' \in S_k^s} 1_{|B_t - B_{t'} | \le 2^{-k}} \le C 2^{2k} k^3,
$$
almost surely.
\end{lemma}

\begin{proof}
  Key to the proof will be a result of Dembo, Peres, Rosen and Zeitouni \cite{DPRZ}. Let $\mu$ denote the occupation measure of Brownian motion at time 1. Then Theorem 1.2 of \cite{DPRZ} states that
  \begin{equation}\label{DPRZ}
  \lim_{\delta \to 0} \sup_{x\in \R^2} \frac{\mu(D(x,\delta))}{\delta^2 (\log (1/\delta))^2} =2
  \end{equation}
  almost surely. In particular, there exists $M(\omega) $ such that $\mu(D(x,\delta)) \le M(\omega) \delta^2( \log(1/\delta))^2$.

  Let $\cA$ denote the event where there is some $t \in [0,1]$ such that $\sum_{t' \in S_k^t} 1_{|B_t - B_{t'} | \le 2^{-k}} > C k^3$ where $C= C(\omega)$ is chosen suitably. By L\'evy's modulus of continuity theorem, we know that
  $$
  \sup_{|s-t|\le 2^{-2k}}  |B_s - B_t| \le  \delta :=c 2^{-k} k
  $$
  almost surely for some universal $c>0$ and for all $k$ sufficiently large. On the event $\cA$, we can thus find $x=B_t$ where $\mu(D(x,\delta)) \ge C k^3 2^{-k}$ for $k$ sufficiently large. Choosing $C(\omega) = c M(\omega)$ we see from \eqref{DPRZ} that $\P(\cA) = 0$.
\end{proof}

   Plugging the estimate of Lemma \ref{L:dprz} into \eqref{inter}, we get
\begin{align*}
  \tilde  \E(|X_k(s) - Y_k(s)|^2 )   & \le C \eps^{4 - \gamma^2 - 2} (\log 1/\eps)^3,
\end{align*}
which proves \eqref{expdecay} at least if $\gamma< \sqrt{2}$.

To prove \eqref{expdecay} in the general case ($\gamma<2$), we introduce the set 
$$
\tilde S_k^s = \{t \in S_k^s: h_\eps(B_t) < - \alpha \log (\eps/ R(B_t;D))\}
$$
where $\alpha> \gamma$ is a fixed parameter which will be chosen close enough to $\gamma$ later on, and $R(z;D)$ denotes the conformal radius at the point $z \in D$. We let $T_k^s = S_k^s \setminus \tilde S_k^s$. Then we have
$$
X_k(s) = \frac1{2^{2k}} \sum_{t \in T_k^s} e^{\bar h_k(B_t)} + \frac1{2^{2k}} \sum_{t \in \tilde S_k^s} e^{\bar h_k(B_t)}.
$$
It is easy to show that the first is negligible. If $\tilde \Q$ denotes the law of the exponential tilting of $\tilde \P$ by $e^{\gamma h_k(B_t)}$, i.e., 
$$
\frac{d \tilde \Q}{d \tilde \P}( \omega) = \frac{e^{\gamma h_k(B_t)}}{\tilde \E(e^{\gamma h_k(B_t)})},
$$
then  letting $\sigma^2 = - \log (\eps/ R(B_t;D))$,
\begin{align}
\tilde \Q(h_k(B_t) \in dx) &= \displaystyle \frac{\displaystyle e^{- \frac{x^2}{2 \sigma^2}} e^{\gamma x} \frac{dx}{\sqrt{2\pi  \sigma^2}}}{ \displaystyle \int  e^{- \frac{x^2}{2 \sigma^2}} e^{\gamma x} \frac{dx}{\sqrt{2\pi  \sigma^2}}} = e^{-\frac{(x-m)^2}{2 \sigma^2}} \frac{dx}{\sqrt{2\pi \sigma^2}}\label{tilt}
\end{align}
where $m = \gamma \sigma^2$, and hence the law of $h_K(B_t)$ under $\tilde \Q$ is $\cN(m, \sigma^2)$.
Thus
\begin{align*}
\tilde \E( 1_{t \in T_k^s} e^{\bar h_k (B_t)}) 
&  = \eps^{\gamma^2/2} \tilde \Q( h_k(B_t) > \alpha \sigma^2) \times \tilde \E(e^{\gamma h_k(B_t)})\\
& \le \eps^{\gamma^2/2} \exp( - \frac12(\alpha - \gamma)^2 \sigma^2) R(B_t;D)^{\gamma^2/2}.
\end{align*}
where the bound above is obtained by using standard bounds on the normal tail distribution. This decays exponentially fast with $k$ uniformly in $t\le T$.

Likewise, by conditioning on $\bar h_k(B_t)$, we get that 
$$
\tilde \E \frac1{2^{2k}} \sum_{t \in T_k^s} e^{\bar h_{k+1}(B_t)} \le C \tilde \E \frac1{2^{2k}} \sum_{t \in T_k^s} e^{\bar h_{k}(B_t)} 
$$
where $C < \infty$ depends only on $\gamma$, and thus this tends to 0 exponentially fast.

Define now $\tilde X_k(s) = \frac1{2^{2k}} \sum_{t \in \tilde S_k^s} e^{\bar h_k(B_t)}$ and $\tilde Y_k(s) = \frac1{2^{2k}} \sum_{t \in \tilde S_k^s} e^{\bar h_{k+1}(B_t)}$. We wish to bound $\E((\tilde X_k(s) - \tilde Y_k(s))^2).$ Applying the same reasoning as in \eqref{variance} shows that
$$
\E((\tilde X_k(s) - \tilde Y_k(s))^2) \le \frac{C}{2^{4k}} \sum_{t, t' \in \tilde S_k^s}1_{|B_t - B_{t'}| \le 2^{-k} }\sqrt{\E( e^{2\bar h_k(B_t)})\E( e^{2\bar h_k(B_{t'})})}
$$
Now when $t \in \tilde S_k^s$, using the same reasoning as in \eqref{tilt} but with tilting proportional to $e^{2 \gamma h_k(B_t)}$ instead
$$
\E(1_{t \in \tilde S_k^s} e^{2 \bar h_k(B_t)})= \eps^{\gamma^2} \Q(X < \alpha \sigma^2) \E(e^{2\gamma h_k(B_t)}) 
$$
where $X \sim \cN(2 \gamma \sigma^2, \sigma^2)$. We may if we wish assume that $\alpha < 2 \gamma$, so 
$$
\Q(X< \alpha \sigma^2) \le \exp( - \frac12 (2 \gamma - \alpha)^2 \sigma^2) \le C \exp( \frac12 (2 \gamma - \alpha)^2 \log \eps).
$$
Thus using Lemma \ref{L:dprz} again,
\begin{align*}
\E((\tilde X_k(s) - \tilde Y_k(s))^2)  &\le C \eps^4 \times (\eps^{-2} \log (1/ \eps)^3) \times \eps^{\gamma^2 + \frac12 (2 \gamma - \alpha)^2} \eps^{- 4 \gamma^2/2}\\
& \le C  (\log 1/\eps)^3 \eps^{2+ \frac12 (2 \gamma - \alpha)^2 - \gamma^2}
\end{align*}
Choosing $\alpha $ arbitrarily close to $\gamma$ we find that the exponent of $\eps$ is arbitrarily close to $2 - \gamma^2/2$ which is positive since $\gamma<2$. Thus we can find $\alpha$ close enough to $\gamma$ such that the exponent is positive, in which case \eqref{expdecay} follows. 

As discussed at the beginning of the section, this implies almost sure convergence of $\alpha_\eps(t)$ to a limit $\alpha(t)$ (which might still be identically zero at this stage).
\end{proof}

\section{Proof of Theorem \ref{T:BM}: Nondegeneracy}
\def\R{\mathbb{R}}
Let $r>0$ and let $T_r = \inf \{ t \ge 0: \text{ dist }(B_t, \partial U) \le r\}$.

We will first show that
$$
\P_z \left\{\lim_{\eps \to 0} \int_0^{T_r} e^{ \gamma h_\eps(B_s)} \eps^{\gamma^2/2} ds >0 \right\} >0.
$$
It suffices to show that the integral is bounded in $L^q$ for some $q>1$. Our strategy is inspired by work of Bacry and Muzy \cite{BacryMuzy} on multifractal random measures. 
Since the proof can appear a bit convoluted, we start by explaining what lies behind it. Essentially, the $q$th moment of the integral can be understood as the sum of two terms: one diagonal term which gives the sum of the local contribution of the field at each point, and a cross-diagonal term which evaluates how these various bits interact with one another. Consider a square $S$ in the domain and such that $z\in S$. The strategy will be to slice the square into many squares of sidelength $2^{-m}$, where $m$ will be a large but finite number. The key part of the estimate is to show that the sum of the contributions inside each smaller square is small. To achieve this, we use a scaling argument, as the Gaussian Free Field in a small square can be thought of as a general `background' height plus an independent Gaussian Free Field in the square.

Without loss of generality, we will assume that $z \in S = (0,1)^2$ the unit square, and $D$ contains the square $S'$, where $S' $ is the square centered on $S$ whose sidelength is 3 (i.e., $S' = (-1, 2)^2$). Then it will suffice to check that
$$
\int_0^{\tau} e^{ \gamma h_\eps(B_s)} \eps^{\gamma^2/2} ds \text{ is bounded in $L^q$,}
$$
where $\tau = \inf \{ t \ge 0: B_t \notin S\}$. We will in fact show the slightly stronger statement that
\begin{equation}\label{goalGFF}
\int_0^{T} e^{ \gamma h_\eps(B_s)} \eps^{\gamma^2/2} \indic{B_s \in S} ds \text{ is bounded in $L^q$,}
\end{equation}
where $T = \inf\{ t \ge 0: B_t \notin S'\}$.

\subsection{Auxiliary fields}

Fix a bounded continuous function $\phi: [0, \infty) \to [0, \infty)$ with the properties that $\phi(0) =0, \phi(x) = 0$ whenever $x \ge 1$. Define an auxiliary  centered Gaussian random field $(X_\eps(x))_{x\in \R^d}$  by specifying its covariance
$$
c_\eps(x,y): = \E(X_\eps(x) X_\eps(y))   = \log_+\left(\frac1{|x- y | \vee \eps} \right)+ \phi\left( \frac{|y-x|}{\eps}\right),
$$
where  $\phi$ is a bounded positive definite function: e.g., $\phi(x) = \sqrt{(1- |x|)_+}$, see \cite{Pasnechenko} and the discussion in Example 2.3 in \cite{RobertVargas}). Define also the normalized field to be $\bar X_\eps(x) = \gamma X_\eps(x) - (\gamma^2/2) \sigma_\eps $, with $\sigma_\eps = c_\eps(0,0) =  - \log \eps +1 $, so that $\E(e^{\bar X_\eps(x)}) = 1$. Because we have assumed that $S' \subset D$, it is easy to check that the covariance structure of $X_\eps$ and $\gamma h_\eps$ are very close: more precisely, there are constants $a$ and $b$, independent of $\eps$, such that
\begin{equation}\label{encadr-cov}
c_\eps(x,y) - a \le \E( \gamma h_\eps(x) \gamma h_\eps(y)) \le c_\eps(x,y) + b
\end{equation}
for all $x,y \in S'$. Condition for a moment on the trajectory of the Brownian path $(B_s, s\le T)$, and let $\tilde \E$ denote the corresponding conditional expectation. Define a measure $\mu_\eps$ to be the (random) Borel measure on $[0, T]$ whose density with respect to Lebesgue measure is $e^{\gamma h_\eps(B_s)}  \eps^{\gamma^2/2}, d \in [0,T]$. Note that conditionally given $B$, the process $\gamma h_\eps (B_s)$ is a centered Gaussian process with covariance function $ \eta_\eps( B_s, B_t)$, where $\eta_\eps(x,y)$ is the covariance function of the (unconditional) Gaussian field $(h_\eps(x))_{x \in D}$. By Theorem 2 of Kahane \cite{Kahane86}, we deduce from the right-hand side of \eqref{encadr-cov} that
$$
\tilde \E\left[ (\mu_\eps(0,T))^q \right] \le \tilde \E\left[ \left(\int_0^T e^{Y_\eps(s) - (1/2) \tilde \E(Y_\eps(s))^2} ds \right)^q \right]
$$
where $Y_\eps(s)$ is a Gaussian centered field with covariance $ c_\eps(B_s,B_t) + b$. Thus $Y_\eps(s)$ may be realized as $Y_\eps(s) = \bar X_\eps(B_s) + W$, where $W$ is a fixed independent centered normal random variable of variance $b$.
We deduce that
\begin{align*}
\tilde \E\left[ (\mu_\eps(0,T))^q \right]  &\le \tilde \E ( e^{qW - qb/2} ) \tilde \E [\left(\int_0^{T} e^{ \bar X_\eps(B_s)} \indic{B_s \in S} ds \right)^q]\\
& =  e^{q(q-1) b/2 } \tilde \E [\left(\int_0^{T} e^{ \bar X_\eps(B_s)} \indic{B_s \in S} ds \right)^q].
\end{align*}
Taking expectations,
\begin{equation}
\label{comparaison1}
\E\left[ (\mu_\eps(0,T))^q \right] \le e^{q(q-1)b/2 } \E [\left(\int_0^{T} e^{ \bar X_\eps(B_s)} \indic{B_s \in S} ds \right)^q].
\end{equation}
Reasoning similarly with the left-hand side of \eqref{encadr-cov} gives us
\begin{equation}
\label{comparaison2}
\E\left[ (\mu_\eps(0,T))^q \right] \ge e^{-q(q-1) a/2 } \E [\left(\int_0^{T} e^{ \bar X_\eps(B_s)} \indic{B_s \in S} ds \right)^q].
\end{equation}

The crucial observation about $X_\eps$ (and the reason why we introduce it) is that it enjoys an exact scaling relation, as follows:
\begin{lemma} \label{L:scaling} For $\lambda<1$,
$$
(X_{\lambda \eps}(\lambda x))_{x \in \R^d} =_d (\Omega_\lambda + X_\eps(x))_{x \in \R^d},
$$
where $\Omega_\lambda$ is an independent centered Gaussian random variable with variance $\log1/\lambda$.
\end{lemma}

\begin{proof}
  One easily checks that for all $x, y \in \R^d$, $c_{\lambda\eps}(\lambda x, \lambda y) = \log1/ \lambda + c_\eps(x,y)$.
\end{proof}

We will need the following quasi-monotonicity lemma: 
\begin{lemma}\label{L:mono}
There exists $c>0$ such that the following holds. For all $\eps'< \eps$ and for all $q>1$, if $d= \tilde d_\eps = \sup_{z \in S}\tilde \E_z(\left( \int_0^T e^{\bar X_\eps(B_s)}ds\right)^q)$,
then
$$
\tilde d_{\eps'} \ge c \tilde d_\eps. 
$$
for some universal $c>0$.
\end{lemma}

\begin{proof} Note that if the quantity $ \int_0^T e^{\bar X_\eps(B_s)}ds$ was  a martingale then the conclusion would be obvious since $q>1$ (and hence $x \mapsto |x|^q$ is convex, so Jensen's inequality for conditional expectations applies). But the integral is not necessarily a martingale and so we need a different argument which is based on comparison with another auxiliary field for which the martingale property does hold, using an idea of Robert and Vargas \cite{RobertVargas}. We summarise the argument below.

Let $\theta(x) = e^{- |x|^2 / 2} / 2\pi$ and let $\theta_\eps(x) = \eps^{-2} \theta(x/ \eps)$. Let $\hat \theta (\xi ) = \int_{\R^2}e^{ - 2 i \pi x\cdot \xi} \theta(x)dx $ be the Fourier transform of $\theta$ and note that $\hat \theta$ is a decreasing function. Then letting $f(x) = \log_+(1/x)$ and $g(t,\xi) = \sqrt{- \hat \theta'(t |\xi|) |\xi|}$, note that $\int_\eps^\infty g(t, \xi)^2 dt = \hat \theta ( \eps |\xi|)$ for all $\xi \neq 0$. 

Therefore if we define
$$
Y_\eps(x) = \int_{(\eps, \infty) \times \R^2} \zeta(x, \xi) \sqrt{ \hat f(\xi)} g(t, \xi) W(dt, d\xi)
$$
where $\zeta(x, \xi) = \cos( 2\pi x \cdot \xi) - \sin(2\pi x\cdot \xi)$ and $W(dt, d\xi)$ is a space-time white noise, we find that $Y_\eps$ is a Gaussian field with covariance
$$
\E(Y_\eps(x) Y_\eps(y)) = \int_{\R^d} \cos( 2\pi(x-y)\cdot \xi) \hat f(\xi) \hat  \theta (\eps |\xi|) d\xi $$
since $f$ and $\theta$ are both radially symmetric and hence $\hat f$ and $\hat \theta$ are even, and hence
$$
\E(Y_\eps(x) Y_\eps(y)) = ( f \star \theta^\eps)(x-y)
$$
by Fourier inversion. Moreover, for any given Borel subset $S$, it is then easy to check that 
$$
\int_0^T 1_{\{B_s \in S\}} e^{\gamma Y_\eps(B_s) - \gamma^2 \var Y_{\eps}}  ds
$$
forms a (reverse) $\tilde \P$-martingale as a function of $\eps$, with respect to $\cF_\eps = \sigma( W(A, \xi), A \subset (\eps, \infty), \xi \in \R)\vee \sigma(B_s, s\le T)$. 
Therefore by Jensen's inequality, since $q>1$, we can now say that
\begin{equation}\label{martpropY}
\tilde \E_z ( \left(\int_0^T 1_{\{B_s \in S\}} e^{\gamma Y_{\eps'}(B_s) - \gamma^2 \var Y_{\eps'}}  ds\right)^q ) \le \tilde \E_z ( \left(\int_0^T 1_{\{B_s \in S\}} e^{\gamma Y_{\eps'}(B_s) - \gamma^2 \var Y_{\eps}}  ds\right)^q ).
\end{equation}
Furthermore, note that there exists constants $c_1, c_2$ such that for $z = x-y$, 
$$
 \log_+ (1/|z|)  - c_1 \le ( f\star \theta_\eps )(z) \le \log_+ (1/|z|) + c_2.
$$
Thus, applying again Theorem 2 of Kahane \cite{Kahane86} yields
$$
c'_1 \tilde d_\eps \le \tilde \E_z ( \left(\int_0^T 1_{\{B_s \in S\}} e^{\gamma Y_{\eps'}(B_s) - \gamma^2 \var Y_{\eps'}}  ds\right)^q ) \le c'_2 \tilde d_\eps.
$$
Lemma \ref{L:mono} follows from \eqref{martpropY}.
\end{proof}

\subsection{Scaling}

Therefore, fix $1<q<2$ and consider for $z\in S = [0,1]^2$ the unit square,
\begin{equation}\label{feps}
f_\eps(z) = \E_z\left[ \left( \int_0^T e^{\bar X_\eps( B_s)}\indic{B_s \in S} ds\right)^q\right],
\end{equation}
where as before, $T= \inf \{ t \ge 0: B_t \notin S'\}$, and $S' = [-1, 2]^2$. We let $M_\eps = \sup_{z \in S} f_\eps(z)$. Our goal will be to show that
\begin{equation}\label{goal}
M_\eps \text{ is uniformly bounded in $\eps$ for some choice of $q>1$}.
\end{equation}

The strategy for the proof of \eqref{goal} will be the following. We fix $m \ge 1$, which we will choose suitably large (but fixed) at some point. We split the square $S$ into a checkerboard pattern of squares $S_i$, each of which has sidelength $2^{-m}$. By Minkowski's inequality, it suffices to show that
$$
\sum_{i \in I} \int_0^T e^{\bar X_\eps(B_s)} \indic{B_s \in S_i} ds
$$
is uniformly bounded in $L^q$, where $(S_i)_{i \in I}$ is a subset of squares such that $|z-w | > 2^{-m}$ for $z \in S_i, w\in S_j$ and $i\neq j \in I$. In words, we have retained ``every other subsquare" in $I$. Note that there are at most $| I | \le 4^m$ such subsquares.

Since $q<2$, the function $x \mapsto x^{q/2}$ is concave and hence subadditive, so, letting $d_i =  \int_0^T e^{\bar X_\eps(B_s)} \indic{B_s \in S_i} ds$,
\begin{align*}
\E_z [\left( \sum_{i \in I} \int_0^T e^{\bar X_\eps(B_s)} \indic{B_s \in S_i} ds \right)^q] \le \E_z [\left( \sum_{i \in I} d_i^{q/2}\right)^2 ]
& = \sum_{i \in I} \E_z(d_i^q) + \sum_{i \neq j} \E_z(d_i^{q/2} d_j^{q/2}).
\end{align*}
We treat separately the diagonal terms and the nondiagonal ones. We start by the diagonal terms.

\begin{lemma}
\label{L:diag}
There exists $C$ independent of $m,q$ and $\eps$ such that
$$
\sum_{i \in I} \E_z(d_i^q) \le Cm 2^{m \zeta(q)} M_\eps
$$
where
\begin{equation}\label{zeta}
\zeta(q) : =  q^2 \frac{\gamma^2 }2  - q (2+ \frac{ \gamma^2 }{2}) +2.
\end{equation}
\end{lemma}

\begin{proof}
Let $T_i = \inf\{ t \ge 0: B_s \notin S'_i\}$, where $S'_i$ is the square centered on $S_i$ containing the 8 adjacent dyadic squares of same size as $S_i$. Let $N_i$ denote the number of times that the path of the Brownian motion returns to $S_i$ after having touched the boundary of $S'_i$.

Then 
applying the Markov property at each such return, we get
$$
\max_{z \in S} \E_z(d_i^q ) \le C\max_{w \in S_i} \E_w(\beta_i^q) \E_z (N_i),
$$
where
$$
\beta_i = \int_0^{T_i} e^{\bar X_\eps(B_s)} \indic{B_s \in S_i} ds.
$$
Now, a simple martingale argument shows that for some constant $C>0$,
$$
\E_z(N_i) \le C \log (2^m),
$$
uniformly in $z \in S$ and $i \in I$.

We now use the scaling properties of both $B$ and $X_\eps$ to estimate the diagonal terms. Let $\lambda = 1/2^m$, and write $2^m B_s = \tilde B_{s 2^{2m}}$. For $w \in S_i$,
\begin{align*}
\E_w( \beta_i^q) & = \E_w[ \left( \int_0^{T_i} \indic{2^m B_s \in 2^m S_i} e^{\gamma \Omega_\lambda + \gamma X_{\eps 2^m} (2^m B_s)} \eps^{\gamma^2/2} ds\right)^q]\\
& = \E( e^{q \gamma \Omega_\lambda}) 2^{-2qm} \E_{\tilde w}[\left(\int_0^{\tilde T_i} \indic{\tilde B_u \in \tilde S_i} e^{\gamma X _{\eps2^m} (\tilde B_u) } (2^m \eps)^{\gamma^2 /2} \right)^q] 2^{-m q\gamma^2 /2}\\
& \le Ce^{q^2 \gamma^2 \log (2^m)/2 } 2^{-2qm - qm \gamma^2 /2} \E_{\tilde w}[\left(\int_0^{\tilde T_i} \indic{\tilde B_u \in \tilde S_i} e^{\gamma X _{\eps2^m} (\tilde B_u) } (2^m \eps)^{\gamma^2 /2} \right)^q] 2^{-m q\gamma^2 /2}
\end{align*}
for a constant $C$ that doesn't depend on $m,q$ or $\eps$. 
By Lemma \ref{L:mono}, this expectation can only increase if we replace $\eps 2^m$ by $\eps$. 

Since there are at most $|I| = 4^m$ terms, we deduce that the contribution of the diagonal terms is at most
$$
\sum_{i\in I} \E_z(d_i^q) \le  C m 2^{m \zeta(q)  } M_\eps,
$$
where $\zeta(q)$ is defined in \eqref{zeta}. This finishes the proof of Lemma \ref{L:diag}.
\end{proof}

\subsection{Interaction term}

We now look at the cross-diagonal terms.  

\begin{lemma}\label{L:cross} There exists $C_{m,q}$ which may depend on $m$ and $q$ but not $\eps$, such that
$$
\sum_{i \neq j} \E_z(d_i^{q/2}d_j^{q/2}) \le C_{m,q}.
$$
\end{lemma}

\begin{proof}
Since $q/2 \le 1$, by H\"older's inequality,
$$
\E_z(d_i^{q/2} d_j^{q/2}) \le \E_z(d_i d_j)^{q/2}.
$$
Now, let $\tilde \E$ denote the conditional expectation given $(B_s, s\le T)$. Then by Fubini's theorem
\begin{align*}
\tilde \E( d_i d_j)& = \int_0^T \int_0^T \tilde \E (  \exp( \bar X_\eps( B_s) + \bar X_\eps(B_t) ) \indic{B_s \in S_i, B_t \in S_j}ds dt \\
& = \int_0^T \int_0^T \eps^{\gamma^2} \exp( \frac{-2\gamma^2 \log \eps + 2 \log_+ (1/\eps \vee |B_s - B_t|) }2) \indic{B_s \in S_i, B_t \in S_j}ds dt\\
& \le C_{m,q} T^2
\end{align*}
Hence, taking expectations, $\E_z(d_i d_j) \le C_{m,q} \E_z(T^2) \le C_{m,q}.$ Taking the $(q/2)$th power, and summing over $i \neq j$, we get Lemma \ref{L:cross}.
\end{proof}

Putting together these two lemmas, we immediately obtain
\begin{equation}
\label{cross}
M_\eps \le C m 2^{m \zeta(q)} M_\eps + C_{m,q}.
\end{equation}
The key fact is that the exponent $\zeta(q)$ may be chosen to be negative for some $q>1$. Indeed,
note that $\zeta(1) = \gamma^2 /2 -2 - \gamma^2 /2  +2= 0$, and $\zeta'(1) = (\gamma^2/2)  - 2 <0$ if and only if $\gamma^2 < 4$ i.e. $\gamma <2$. Since this is the assumption of the theorem, we deduce $\zeta'(1) < 0 $ and hence it is possible to choose $q>1$ such that $\zeta(q) <0$.

Since $\zeta(q) < 0$, we can choose $m$ sufficiently large that $Cm 2^{m \zeta(q)} < 1/2$, and so obtain that
$
M_{\eps} \le 2 C_{m,q}.
$
This proves \eqref{goal}, and therefore \eqref{goalGFF}. Thus $\alpha(T_r) >0$ with positive probability.

\subsection{Continuity}

 Let $\mu_\eps$ denote the random Borel measure on $\R$ obtained by
$$
\mu_\eps((s,t]) = \int_s^t e^{\gamma  h_\eps(B_u)} \indic{u < T} \eps^{\gamma^2 / 2} du.
$$
Then by the first part of the argument, $\mu_\eps$ converges to a measure $\mu$ and we have just shown that $\mu_\eps(0,\infty)$ is bounded in $L^q$ for some $q>1$, hence is uniformly integrable. Thus $\E(\mu(0,\infty)) = 1$ and thus $\mu$ is positive at least with positive probability. We now check that this probability must in fact be equal to one.

Using the scaling property of $\bar X$ as in Lemma \ref{L:diag} in combination with \eqref{goal}, we see that there exists $C>0$ such that for any dyadic square $S_i \subset S$ of sidelength $2^{-m}$ and $z \in S_m$ is arbitrary, then
$$
\E_z\left[ \left(\int_0^{\tau(S_i)} e^{\bar X_\eps( B_s) } ds\right)^q \right] <C m 2^{m (\zeta(q) -2)}.
$$
Consequently, for some $q>1$ (since we have boundedness in $L^{q'}$ for some $q'>q$), for any $ z \in S_i$,
\begin{equation} \label{1square}
\E_z \left[ \left(\int_0^T \indic{ B_t \in S_i } d\mu (t)\right)^q \right] <C m 2^{m (\zeta(q) -2)}.
\end{equation}
In fact, by using the Markov property, this also holds even for arbitrary $z \in S$. Define the event $$G_i = \left\{\int_0^T \indic{ B_t \in S_i } d\mu (t) >0\right\}.$$
We wish to show that many $G_i$ occur with high probability. That is, our goal is to show that if $Z_m = \sum_i \indic{G_i}$, where the sum is over all dyadic squares $S_i$ of sidelength $2^{-m}$, then $\P( Z_m >0) \to 1$ as $m \to \infty$.
\begin{lemma}
  \label{E(Zi)} There exists $c_q>0$ such that for any $m \ge 1$ and any $q>1$ sufficiently close to 1,
  $$
  \E(Z_m) \ge  c2^{m(4 - \zeta(q))} m^{-2}.
  $$
\end{lemma}

\begin{proof}
We simply note that by H\"older's inequality, $$\P_z(G_i) \ge c_q \frac{\E_z\left[\int_0^{\tau(S_m)} \indic{B_t \in S_i} d \mu(t)\right]}{ \E_z\left[ \left(\int_0^{\tau(S)} \indic{B_t \in S_i} d\mu(t) \right)^q \right]}.$$
By uniform integrability of $\mu_\eps$, 
$$
\tilde \E\left[\int_0^{\tau(S)} \indic{t \in S_i} d \mu(t)\right] =\int_0^{\tau(S)} \indic{B_t \in S_i} dt.
$$
Taking expectations,
$$
\E\left[\int_0^{\tau(S)} \indic{t \in S_i} d \mu(t)\right] =\E[\int_0^{\tau(S)} \indic{B_t \in S_i} dt] \ge c/m.
$$
Thus the result comes from \eqref{1square}.
\end{proof}

We now show the second moment estimate needed to conclude:

\begin{lemma}
  \label{2dmoment}
$$
\var(Z_m) \le  \E(Z_m).
$$
\end{lemma}

\begin{proof}
  Assume that $S_i$ and $S_j$ are two disjoint squares of sidelength $2^{-m}$. 
  Using the Markov property of the Gaussian Free Field (see e.g. the statement of Proposition 2.3 in \cite{HMP}), we see that conditionally given the values of $h|_U$ where $U = S\setminus (S_i \cup S_j)$,  we can write
  $$
  h = h_U + h^i + h^j
  $$ 
  where $h_U$ is the harmonic extension of $h|_U$ to $S$ (which is a.s. a harmonic function on $S$) and $h^i, h^j$ are independent Gaussian Free Fields with zero boundary condition on $S_i, S_j$, and independent of $h_U$. Then note that the event $G_i$ is a function of $h_i$ solely.  
  Hence  $G_i$ and $G_j$ are in fact independent events. The result follows.
  \end{proof}
To finish, observe that by Chebyshev's inequality:
\begin{align*}
\P(\mu[0,1]>0) &\ge \P( \cup G_i )  = \P(Z_m >0)\\
& \ge 1- \frac{\var(Z_m)} {\E(Z_m)^2} \ge 1 - \frac1{\E(Z_m)}\\
& \ge 1- \frac{m^2}{c 2^{m (4 - \zeta(q))}} \to 1
\end{align*}
as $m \to \infty$. Thus $\P(\mu(0,t) >0) = 1$ for all $t >0$. It follows from this that, almost surely, for all rationals $s < t$,
$$
\mu ((s,t]) >0.
$$
Hence, since $\mu$ is nonnegative, this is also true for all times $s<t$ simultaneously. 

Therefore $t\mapsto \mu^{-1}(t)$ is continuous with probability one, and so $t\mapsto Z(t) = B_{\mu^{-1}(t)}  $ is also continuous with probability one.

\section{Conformal Invariance}

Naturally, the Gaussian Free Field is conformally invariant as a random distribution. However, its regularisation $h_\eps$ is not, and so it is better to consider a different approximation of the Gaussian Free Field. Fix $f_1, \ldots$ an orthonormal basis of $H_0^1(D)$, say by considering normalised eigenvectors of $- \Delta$ with Dirichlet boundary conditions on $\pd D$.

Let $h^n (z) = \sum_{i=1}^n X_i f_i$, where $X_i$ are i.i.d. standard normal random variables, and note that we can think of $h^n$ as the orthogonal projection of $h$, which is formally the infinite sum $\sum_{i=1}^\infty X_i f_i$, onto $\text{Span}(f_1, \ldots, f_n)$. 

Define 
\begin{equation}\label{hn}
\mu^n([0,t]) = \int_0^t \indic{t < T} e^{\gamma h^n(B_t) - \frac{\gamma^2}2 \var h^n(B_t) + \log R(B_t;D)} dt.
\end{equation}

The following proposition shows that approximating $h$ by $h_n$ does not change the limiting diffusion.
\begin{prop}Almost surely for all $t \ge 0$,
$$
\mu^n([0,t]) \to \mu ([0,t])
$$
as $n \to \infty$.
\end{prop}
\begin{proof}
Define $h_n^\eps(z) $ to be the average of $h^n(w)$ on a circle of radius $\eps$ about $z$.  Then for each fixed $\eps>0$, the sequence $e^{\gamma h_\eps^n(z) - \frac{\gamma^2}2 \var h^n_\eps(z) }$ forms a nonnegative martingale with respect to $n$, and the filtration $\cH_n = \sigma(h_i, 1\le i \le n)$. The limit as $n \to \infty$ is naturally 
$e^{\gamma h_\eps(z) - \frac{\gamma^2}2 \var h_\eps(z) }$, which also has expectation equal to 1. Thus the martingale is uniformly integrable and
we have
$$
\E ( \mu_\eps([0,t]) | \cH_n)   =  \int_0^t \indic{t < T} e^{\gamma h_\eps^n(B_t) - \frac{\gamma^2}2 \var h^n_\eps(B_t) + \frac{\gamma^2}2 \log R(B_t;D)} dt.
$$
Thus letting $\eps \to 0$, since $h_n$ and $\var (h_n)$ are continuous, 
$$
\lim_{\eps \to 0} \E ( \mu_\eps([0,t]) | h^n)  = \mu^n[0,t].
$$
But by Fatou's lemma,
$$
\E( \mu([0,t]|\cH_n) = \E( \liminf \mu_\eps([0,t]|\cH_n) \le \liminf_{\eps \to 0} \E ( \mu_\eps([0,t]) | h^n)  = \mu^n[0,t].
$$
Hence, for all $t$, and all $n$, almost surely,
$$
\E( \mu([0,t])|\cH_n) \le \mu_n([0,t]).
$$
But taking expectations, the left hand side is equal to $\E \int_0^t  \indic{t<T}( \log R(B_t;D))^{\frac{\gamma^2}2} dt$ as $\mu_\eps$ is uniformly integrable, and the right hand side is also equal to the same value. Since these two random variables are almost surely ordered and have the same expectation, they are almost surely equal. 

We deduce
$$
\E( \mu([0,t])|\cH_n) = \mu^n([0,t]).
$$
By the martingale convergence, we deduce that $\mu^n([0,t]) \to \mu ([0,t]) $ as $n \to \infty$, almost surely. 
\end{proof}

Now, let $\phi : D \to \tilde D$ be a conformal transformation and let $\psi = \phi^{-1}$. Then writing $\tilde f_n = f_n \circ \phi$, we see that $\tilde f_n$ forms an orthonormal basis of $H_0^1 (\tilde D)$ (this is because $(\cdot, \cdot)_\nabla$ is conformally invariant). Thus let $\tilde h^n = h^n \circ \phi$, which is the projection of the Gaussian Free Field $h \circ \phi$ onto $\text{Span}(\tilde f_1, \ldots, \tilde f_n)$. Let $\mu_n = \mu^n$ and $\mu_n^{-1}$ be the inverse function of $\mu_n$. Now by conformal invariance of ordinary Brownian motion,
$$
\phi(B_t) = \tilde B_{\int_0^t |\phi'(B_s)|^2} ds
$$
where $\tilde B$ is a killed Brownian motion in $D$. Thus
$$
\phi(B_{\mu_n^{-1}(t)}) = \tilde B_{\sigma_n(t)} =: \tilde Z_n(t)
$$
where, by definition, $\sigma_n(t) = \int_0^{\mu_n^{-1} (t)} |\phi'(B_s)|^2 ds.$ By the chain rule, if $\tilde z = \tilde Z_n(t)$ and $z = B_{\mu_n^{-1}(t)}$, 
\begin{align*}
\sigma_n'(t) &= \frac{d}{dt}\mu_n^{-1} (t) |\phi'(B_{\mu^n(t)})|^2\\
& =\frac1{ e^{\gamma  h^n (z) - \frac{\gamma^2}2 \var h^n(z) + \frac{\gamma^2}2 \log R(z;D) } | \psi'(\tilde z)|^2}
\end{align*}
Observe that $ \log R(\tilde z; \tilde D) = \log R( z, D) + \log |\phi'(z)|$ and hence
\begin{equation}\label{timechange}
\sigma'_n(t) = \frac1{ e^{\gamma  (\tilde h^n ( \tilde z) + Q \log |\psi'|(\tilde z)) - \frac{\gamma^2}2 \var \tilde h^n(\tilde z) + \frac{\gamma^2}2 \log R(\tilde z;\tilde D) } },
\end{equation}
where $Q = \gamma /2 + 2/ \gamma$. Thus define a field in $\tilde D$ by $\tilde h_\psi(w) = h\circ \psi + Q \log |\psi'|$. 

Then the right hand side is the derivative of $\tilde \mu^n_\psi(t)^{-1} $, where $$\tilde \mu_\psi^n (t) = \int_0^t  e^{\gamma h^n_\psi(\tilde B_s) - \frac{
\gamma^2}2 \var \tilde h_\psi^n(\tilde B_s) - \frac{\gamma^2}2 \log R(\tilde B_s; \tilde D)  } ds.$$
Hence we have proved, after taking limits as $n\to \infty$,
$$
\phi(Z_t) = \tilde B_{\tilde \mu_\psi^{-1} (t)}.
$$

\section{Proof of Theorem \ref{T:thick}}

We focus on the case $\alpha > \gamma$ and consider the set $\{t: Z_t \in \cT^+_\alpha\}$ (the other case is identical). To ease the proof we will drop the superscript $+$ from this notation and hence call $\cT_\alpha : = \cT^+_\alpha$ in this proof. Let $\delta, \zeta>0$ be fixed and fix $\eta>0$. We set $K = 3 /( \eta ( 2 - \alpha^2 /2 ))$ and choose $r_n = n^{-K}$ a sequence of scales. 
Let $t_{nj} = j r_n^2$, $1\le j \le r_n^{-2}$ form a partition of $[0,1]$ into intervals of size $r_n^2$. 
It will be important to note that $r_n$ depends solely on $\eta$ and that $\delta$ can be as small as desired compared to $\eta$, without affecting the choice of $r_n$. 

If $t_{nj}$ is the closest element of the net to $t$, then
$$
|B(t) - B(t_{nj})| \le C \sqrt{r_n^2 \log(r_n)^2}
$$
by L\'evy's result on the uniform modulus of continuity of Brownian motion (\cite{MortersPeres}). Hence applying Proposition 2.1 in \cite{HMP}, for all $\eps>0$ and $\eta>0$,
\begin{align}
|h_{r_n}(B(t)) - h_{r_n}(B(t_{nj})| &\le \frac{C(\log n )^{\zeta} }{r_n^{\delta/2}} \nonumber \\
& \le C r_{n}^{-\delta}. \label{hcont}
\end{align}

This motivates the following definition
$$
\cI_n = \{j:  h_{r_n}(B(t_{nj}))  \ge (\alpha - \delta) \log (1/r_n)\}.
$$ 
The interest of introducing this set then comes from the fact that if $B_t \in \cT_\alpha$ then we can find arbitrarily large $n$ such that for some $j \in \cI_n$, we have $t \in [t_{nj} - r_n^2, t_{nj} + r_n^2] $.

Now, it is plain that
$$
\E(|\cI_n|) = r_n^{-2} \P( h_{r_n}(z) \ge (\alpha - \delta_n) \log (1/r_n)) \sim r_n^{-2 +( \alpha - \delta)^2/2 + o(1)}.
$$
and thus
\begin{equation}\label{numint}
\E(|\cI_n|) \le r_n^{-2 +\alpha^2/2 - \delta \alpha + o(1)}.
\end{equation}

Recall the inverse clock function $t\mapsto \mu^{-1}(t) $ where with a slight abuse of notation we identify $\mu(t)$ and $\mu([0,t])$. By \eqref{hcont}, mapping the intervals $[t_{nj } - r_n^2, t_{nj} + r_n^2]$ by $t \mapsto \mu^{-1}(t)$ gives rise to a cover of $\{t : Z(t) \in \cT_\alpha\}$. The image of these intervals are also intervals and it suffices to estimate the diameter. 

\begin{lemma}\label{L:covermoments} Given $j \in \cI_n$, let $a_{nj} = \mu^{-1}(t_{nj} - r_n^2) $ and $b_{nj} = \mu^{-1} (t_{nj} + r_n^2)$. Then for all $q\le 1$,
$$
\E( \diam ([a_{nj}, b_{nj}] )^q | j \in \cI_n) \le Cr _n^{q(2 -\gamma \alpha + \gamma^2/2)}.
$$
\end{lemma}

\begin{proof}
By Jensen's inequality it suffices to prove the result for $q=1$. Note that $ \diam ([a_{nj}, b_{nj}] ) = \mu ([t_{nj} - r_n^2, t_{nj} + r_n^2])$. Hence, by uniform integrability of $\mu_\eps$, 
\begin{align*}
\E( \diam ([a_{nj}, b_{nj}] ) | j \in \cI_n)  & = \lim_{\eps \to 0} \E(\mu_{\eps}  [t_{nj} - r_n^2, t_{nj} + r_n^2] | j \in \cI_n) \\
& \asymp C  \int_{t_{nj} - r_n^2}^{t_{nj} + r_n^2} \E  ( e^{\gamma h_\eps(B_s) + (\gamma^2 /2) \log \eps  } | j \in \cI_n) ds\\
& = C  \int_{t_{nj} - r_n^2}^{t_{nj} + r_n^2} \E (e^{\gamma h_{r_n} (B_s) + (\gamma^2 /2) \log r_n } |j \in \cI_n) ds\\
& \asymp \E  [\mu_{r_n}([t_{nj} - r_n^2, t_{nj} + r_n^2])  | j \in \cI_n]
\end{align*}
But it is easy to check that given $j \in \cI_n$, $ \E (e^{\gamma h_{r_n} (B_s) + (\gamma^2 /2) \log r_n }) \asymp r_n^{-\gamma \alpha + \gamma^2/2} $ for all $s \in [t_{nj} - r_n^2, t_{nj} + r_n^2]$.
 Thus
$$
\E( \diam ([a_{nj}, b_{nj}] ) | j \in \cI_n) \asymp r_n^{2 -\gamma \alpha + \gamma^2/2},
$$
from which the result follows.
\end{proof}

Let 
$$
J_N = \cup_{n\ge N} \{ [a_{nj} , b_{nj}], j \in \cI_n\},
$$
and note that for all $N\ge 1$, $J_N$ covers $\{t: Z_t \in \cT_\alpha\}$. 
Now, let $q = d(1+ \eta)$ where $d = (2- \alpha^2/ 2) /  ( 2 - \alpha \gamma  + \gamma^2 /2 ) $, and note that if $\gamma<2$ and $\alpha > \gamma$ then we can choose $\eta>0$ small enough so that $ q < 1$. By Lemma \ref{L:covermoments}, choosing $\delta$ sufficiently small,
\begin{align*}
\E (\sum_{n \ge N} \sum_{j \in \cI_n} \diam ([a_{nj}, b_{nj} ] )^q ) & = O ( \sum_{n \ge N} r_n^{\eta ( 2 - \alpha^2 / 2) - \alpha \delta + o(1)})\\
& = O( \sum_{n \ge N} n^{-2 + o(1)}).
\end{align*}
This proves that the Hausdorff $q$-dimension of $\{t: Z_t \in \cT_\alpha\}$ is 0, almost surely. Since $\eta>0$ is arbitrary, this proves the result.

\medskip \textbf{Acknowledgements.} I am indebted to Gr\'egory Miermont, who mentionned this question and discussed it with me at several stages. I also thank Christophe Garban and Jason Miller for several illuminating conversations about the Gaussian Free Field.

\end{document}